\documentclass[11pt, reqno]{amsart}

\usepackage[top=3.75cm, bottom=3cm, left=3cm, right=3cm]{geometry}
\usepackage{amsthm}
\usepackage{mathscinet}
\usepackage{amsmath}
\usepackage{amsfonts}
\usepackage{amssymb}
\usepackage{mathtools}
\usepackage{bm}
\usepackage{color}
\usepackage{enumitem}
\usepackage{tikz}  
\usepackage{graphicx}
\usepackage{scalerel}
\usepackage{comment}
\usepackage[mathscr]{eucal}
\usepackage{stmaryrd}
\usepackage{appendix}
\usepackage{indentfirst}
\usepackage{tikz-cd}
\usepackage[colorlinks=true, citecolor=citation, urlcolor=citation, linkcolor=reference, linktocpage = true]{hyperref}
\usepackage{cleveref}

\definecolor{citation}{rgb}{0,.40,.80}
\definecolor{reference}{rgb}{.80,0,.40}

\numberwithin{equation}{section}

\theoremstyle{plain}
\newtheorem{theorem}{Theorem}[section]

\newtheorem{lemma}[theorem]{Lemma}
\newtheorem{proposition}[theorem]{Proposition}
\newtheorem{corollary}[theorem]{Corollary}

\theoremstyle{definition}
\newtheorem{definition}[theorem]{Definition}

\newtheorem{remark}[theorem]{Remark}

\setlist[itemize]{topsep=5pt,itemsep=3pt}
\setlist[enumerate]{topsep=5pt,itemsep=3pt}

\DeclareMathOperator{\cEnd}{\mathcal{E}\kern -1.2pt {\it nd}}

\newcommand{\cO}{\mathcal{O}}

\newcommand{\rH}{\mathrm{H}}

\newcommand{\fm}{\mathfrak{m}}

\newcommand{\bA}{\mathbf{A}}
\newcommand{\bC}{\mathbf{C}}

\newcommand{\bZ}{\mathbf{Z}}
\newcommand{\bP}{\mathbf{P}}

\newcommand{\pfister}[1]{\langle\!\langle #1\rangle\!\rangle}

\newcommand{\dra}{\dashrightarrow}

\definecolor{proofreadgreen}{rgb}{0,0.45,0}
\newcommand{\proofreadcomment}[3]{%
  \par\smallskip
  {\normalfont\small\noindent
   {\color{blue}\textbf{Comment #1.} #2}%
   \if\relax\detokenize{#3}\relax\else
     \par\noindent{\color{proofreadgreen}\textbf{Suggested fix #1.} #3}%
   \fi\par}
  \smallskip
}

\begin{document}

	\title[The rationality problem for hypersurfaces of degree at least five]{The rationality problem for hypersurfaces \\ of degree at least five}
    \author{James Hotchkiss}

    \begin{abstract}
        We prove that a very general hypersurface of degree at least five and any positive dimension does not have a decomposition of the diagonal, over an uncountable algebraically closed field of characteristic not $2$.
        In particular, such hypersurfaces are not stably or retract rational.
        For the proof, we construct (in each dimension $\geq 8$) a certain quintic hypersurface with a nonzero unramified cohomology class, and which is rationally fibered by quadrics over a lower-dimensional projective space.
        From our construction, the main result follows using methods developed by Schreieder.
    \end{abstract}
    
    \maketitle

    \section{Introduction}

    Let $k$ be an algebraically closed field.
    A smooth projective variety $X$ over $k$ is \emph{rational} if it is birational to projective space, \emph{stably rational} if $X \times \bP^s_k$ is rational for $s \gg 0$, and \emph{retract rational} if there is a dominant rational map $\bP^s_k \dra X$ with a rational section.
    We note that
    \begin{equation*}
        \textrm{rational} \implies \textrm{stably rational} \implies \textrm{retract rational}.
    \end{equation*}
    All three properties imply that $X$ has a decomposition of the diagonal. 
    A basic problem in algebraic geometry is the question of which hypersurfaces in projective space are rational, stably rational, or retract rational.
    Our main result is the following:
    
     \begin{theorem}
    \label{thm:main-intro}
        Suppose that $k$ is uncountable, and the characteristic of $k$ is not $2$. 
        The very general hypersurface $X$ of degree $\geq 5$ and any positive dimension does not have a decomposition of the diagonal.
        In particular, $X$ is not retract rational, stably rational, or rational.
    \end{theorem}

    The rationality problem for very general hypersurfaces has a long history; we briefly sketch it here, and refer to the recent survey \cite{schreieder-icm} for further details.
    First, it is widely believed that the very general hypersurface $X$ of degree $d \geq 3$ and dimension $n \geq 3$ is not retract rational; equivalently, the only very general hypersurfaces which are retract rational are those which are obviously so: hyperplanes, quadrics, and cubic surfaces.
    (Theorem~\ref{thm:main-intro} reduces this question to the case of cubics and quartics.)
    In the last three decades, there has been remarkable progress, which we summarize mostly over $k = \bC$ for simplicity: Koll\'ar \cite{kollar1995} showed that $X$ (assumed here and below to be very general) is not ruled, hence irrational, when $d \geq 2 \lceil (n + 3)/3 \rceil$. 
    Totaro \cite{Totaro2016} applied the specialization method from \cite{Voisin2015, ColliotThelenePirutka2016} to show that $X$ is stably irrational when $d \geq 2 \lceil (n + 2)/3 \rceil$.
    In a breakthrough, Schreieder \cite{schreieder} improved the bound to
    \begin{equation*}
        d \geq \log_2(n) + 2,
    \end{equation*}
    which also holds in positive characteristic not $2$.
    The improvement is based on a powerful version of the specialization method introduced in \cite{Schreieder2019QuadricBundles}, as well as innovative methods from the theory of quadratic forms (for which \cite{HPT2018} was an important precursor).
    There have been more recent improvements to Schreieder's bound by Moe and Nicaise--Ottem \cite{moe, nicaise-ottem} in characteristic $0$, and Lange--Schreieder \cite{lange-schreieder} in characteristic not $2$; in a different vein, it was shown recently in \cite{EngelDeGaayFortmanSchreieder2025} that the very general cubic threefold over the complex numbers is stably irrational, and very recently it was shown in \cite{Ottem2026} that a very general quartic sixfold is not retract rational.
    Theorem~\ref{thm:main-intro} is the first result in which the range of degrees does not depend on the dimension.

    \subsection{Overview of the proof}

    The proof of Theorem~\ref{thm:main-intro} involves  a construction, which we feed into general machinery developed by Schreieder over a series of papers, \cite{Schreieder2019QuadricBundles, schreieder, schreieder-torsion}.
    For the purposes of analyzing both the construction and Schreieder's machinery, we consider the notion of a \emph{Schreieder triple}, which is implicit in Schreieder's work.
    Loosely speaking, a Schreieder triple $(\alpha, Q, H)$ consists of a smooth projective quadric $Q$ over the function field $K$ of a projective space $\bP^r_k$ over $k$, a top-degree symbol $\alpha \in \rH^{r}(K, \bZ/2)$ over $K$ which pulls back to a nonzero cohomology class $\alpha_{K(Q)}$ on $K(Q)$ which is unramified over $k$, and a hyperplane $H$ such that $\alpha_{K(Q)}$ vanishes on $K(Q \cap H)$.

    \subsubsection{General machinery}
    \label{sssec:general-machinery}
    From any Schreieder triple, one can spread out the equation for the quadric to obtain a (potentially highly singular) hypersurface over $k$, with a nonzero unramified cohomology class. 
    Because of the inexplicit nature of the singularities, it is not clear that one can apply the original form of the specialization method from \cite{Voisin2015, ColliotThelenePirutka2016}, which relies on analyzing a resolution of singularities.
    However, a key improvement to the specialization method first introduced in \cite{Schreieder2019QuadricBundles} often allows one to circumvent the problem of computing a resolution.

    It turns out that for an arbitrary Schreieder triple, the methods from \cite{Schreieder2019QuadricBundles, schreieder, schreieder-torsion} apply and allow one to obtain a bound on the retract irrationality of hypersurfaces of a certain degree and dimension, cf. Theorem~\ref{thm:hypersurfaces-from-schr-trip}.
    For this reason, we focus our energy entirely on the problem of constructing Schreieder triples which spread out to singular hypersurfaces of degree $5$.

    \subsubsection{Sketch of the construction} 
    \label{ssec:intro-the-construction}
    For simplicity, we just describe the construction which leads to Theorem~\ref{thm:main-intro} for even-dimensional hypersurfaces; the odd-dimensional case requires a slight modification.
    We begin in an extremely na\"ive fashion.

    We start with the symbol $\alpha_r = (z_1, \dots, z_r) \in \rH^r(K_r, \bZ/2)$ over the field $K_r = k(z_1, \dots,z_r)$, and consider the smooth projective quadric $Q_0$ associated to the quadratic form 
    \begin{equation}
    \label{eq:intro-original-form}
        \langle 1, 1, z_1, \dots, z_r \rangle
    \end{equation}
    over $K_r$.
    The good news is that the pullback of $\alpha_r$ to $K_r(Q_0)$ remains nonzero. 
    Indeed, this follows from the fact that \eqref{eq:intro-original-form} is isotropic, so that $K_r(Q_0)$ is purely transcendental over $K_r$.
    The bad news is that, for essentially the same reasons, the pullback of $\alpha_r$ is ramified, so there is no hope of fitting the pair $(\alpha_r, Q_0)$ into a Schreieder triple.
    
    The key point is that for a parameter $t \in k$, and under the assumption that $r \geq 3$, one may construct a perturbation 
    \begin{equation}
    \label{eq:intro-perturb-form}
        \langle 1, 1 + t^2 f_0(t), z_1 + t^2 f_1(t), \dots, z_r + t^2 f_r(t)\rangle, \quad f_0(t), \dots, f_r(t) \in K_r(t)
    \end{equation}
    of \eqref{eq:intro-original-form}, with each $f_j(t)$ regular at $t=0$, so that if $Q_t$ is the corresponding quadric, then the pullback of $\alpha_r$ to $K_r(Q_t)$ is unramified for $t \neq 0$.
    By a result of Gabber from \cite[Appendix]{CT-examples}, the pullback of $\alpha_r$ to $K_r(Q_t)$ remains nonzero for $t$ very general (since it is nonzero at $t = 0$), and it is not difficult to show that the pair $(\alpha_r, Q_t)$ fits into a Schreieder triple. 

    The final point to address is the degree of the resulting hypersurface. 
    For this, the crucial idea is that for the particular perturbation constructed above, after a  rational change of coordinates, the Gram matrix of \eqref{eq:intro-perturb-form} has entries which are polynomials of degree $\leq 3$ in the variables $z_1, \dots, z_r$.
    In these better coordinates (which in the body of the paper correspond to the basis vectors $v_0 + tp_0, \dots, v_r + tp_r$ of the form $\sigma_t$ considered in \S \ref{sec:construction}), the quadric $Q_t$ spreads out to a singular quintic hypersurface in $\bP^{2r + 1}_k$, and we conclude by applying the general machinery described in \S\ref{sssec:general-machinery}.

    \subsection{Outline of the paper}

    In \S \ref{sec:schr-trip}, we introduce the notion of a Schreieder triple, explain the result of Schreieder connecting them to the rationality problem for hypersurfaces, and discuss some important criteria for proving that a symbol is unramified. 
    In addition, we describe a simple but key construction in \S \ref{ssec:augmented-subform}.
    In \S \ref{sec:construction}, we explain the construction of the particular family of Schreieder triples which goes into Theorem~\ref{thm:main-intro}.
    In \S \ref{sec:conclusion}, we put everything together and prove Theorem~\ref{thm:main-intro}.

    \subsection{Notation and conventions}

    Throughout, $k$ is an uncountable algebraically closed field of characteristic not $2$.
    Given a field $K$ containing $k$ and elements $a_1, \dots, a_n \in K^\times$, we write $(a_1, \dots, a_n) \in \rH^n(K, \bZ/2)$ for the mod $2$ symbol.

    All quadratic forms are assumed to be  nondegenerate.
    If $\sigma$ is a quadratic form, our convention is that a \emph{subform} $\tau \subset \sigma$ consists of a subspace of the vector space underlying $\sigma$, such that the restriction of $\sigma$ to this subspace is nondegenerate.

    \subsection{Acknowledgements} 
    We are very grateful to Stefan Schreieder, both for his comments on a draft paper, his advice, and for his inspiring work in the subject. 
    It should be clear to the reader that this paper  exists first and foremost because of a remarkable set of techniques and insights developed by Schreieder in the course of his work on the rationality problem.
    In addition, we would like to thank Alex Perry for helpful advice. 

    \subsection{AI disclosure}
    \label{ssec:ai-disclosure}

    The project was developed from a sustained interaction between the author and ChatGPT-5.6 Sol and later ChatGPT-6 Astra over the course of one month, with the goal of improving  the bound on the stable irrationality of very general hypersurfaces from \cite{lange-schreieder}.
    During that time, the author used ChatGPT to search for and develop a number of constructions based on the methods from \cite{moe} and \cite{lange-schreieder}, first in the combinatorial framework of \cite{nicaise-ottem} and later in the framework of Schreieder triples detailed in \S \ref{sec:schr-trip}.
    
    We emphasize that the final construction,  described in \S \ref{ssec:intro-the-construction} above and \S \ref{sec:construction} below, is due to ChatGPT-6 Astra.
    In addition, ChatGPT-6 Astra is responsible for suggesting almost all of the proofs, although the details have been written up here by the author.

    During the drafting process, ChatGPT-6 Astra was used to assist with proofreading and to help propagate changes in notation.
    Otherwise, this paper was written entirely by the author, who is responsible for its correctness.

    \section{Schreieder triples}
    \label{sec:schr-trip}

    \subsection{The definition}
    \label{ssec:the-definition}

    Throughout the paper, we freely use the basics of unramified cohomology \cite{CTO1989},
    for which we use the survey \cite{schreieder-survey} as a reference.
    For example, we will repeatedly use the fact that if $K/k$ is a finitely generated field extension, $v$ is a discrete valuation on $K$, and $\alpha = (a_1, \dots, a_n) \in \rH^n(K, \bZ/2)$ is a symbol over $K$ such that $v(a_1) = \cdots = v(a_n) = 0$, then $\alpha$ is unramified at $v$.
    This fact follows easily from the fact that each $a_i$ lies in the subgroup $\cO_v^\times \subset K^\times_v$, so over $K_v$ one directly lifts $\alpha$ to an element of $\rH^n(\cO_v, \bZ/2)$.
    Therefore, the residue of $\alpha$ vanishes over $K_v$, hence one concludes that $\alpha$ is unramified at $v$ using the compatibility of residues with completion \cite[\S 3.1, (3.3)]{schreieder-survey}.

    \begin{definition} 
        A \emph{Schreieder triple} is a triple $(\alpha, Q, H)$, where $\alpha = (\alpha_1, \dots, \alpha_b) \in \rH^b(K, \bZ/2)$ is a top-degree symbol over $K = k(x_1, \dots, x_b)$, $Q \subset \bP_K^s$ is a smooth integral projective quadric, and $H \subset \bP_K^s$ is a hyperplane, satisfying the following conditions:
        \begin{enumerate}
            \item The pullback of $\alpha$ to the function field $K(Q)$ is nonzero and unramified;
            \item The intersection $Q_H  \coloneqq Q \cap H$ is smooth over $K$ and integral, and the restriction of $\alpha_{K(Q)}$ to $K(Q_H)$ vanishes.
        \end{enumerate}
    \end{definition}

    By choosing coordinates, we can identify an affine chart $\bA_K^s \subset \bP_K^s$ with the generic fiber of a projection $\bP_k^{b+s} \dra \bP_k^b$. 
    In particular, the closure of $Q \cap \bA_K^s$ in $\bP_k^{b+s}$ is a hypersurface $X$ over $k$ of some degree $d$. 
    One of the insights from \cite{schreieder} is that if one takes the chart $\bA_K^s$ to be the complement of the hyperplane $H$ occurring in the Schreieder triple, then the pullback of $\alpha$ to the function field of $X$ obstructs the existence of a decomposition of the diagonal on a very general deformation of $X$.
    More precisely, one has the following result, which has been extracted more or less directly from Schreieder's work in \cite{schreieder} and \cite{schreieder-torsion}.

    \begin{theorem}[Schreieder]
    \label{thm:hypersurfaces-from-schr-trip}
        Let $(\alpha, Q, H)$ be a Schreieder triple, and identify the affine chart $U_H \coloneqq \bP_K^s \setminus H$ with the generic fiber of a projection $\bP_k^{b+s} \dra \bP_k^b$ as above.
        Let $d$ be the degree of the closure of $Q \cap U_H$ in $\bP_k^{b+s}$. Then a very general hypersurface over $k$ of degree $d' \geq d$ and dimension $b + s - 1$ does not admit a decomposition of the diagonal, and in particular is stably irrational.
    \end{theorem}

    \begin{proof}
        Let $\Lambda \subset \bP_k^{b + s}$ be the plane of codimension $b + 1$ from which the projection $\bP^{b + s}_k \dra \bP^b_k$ is defined. 
        We choose coordinates so that
        \begin{align*}
            [Y_0, Y_1, \dots, Y_s] &\textrm{ on } \bP^{s}_K \\
            [Y_0, Y_1, \dots, Y_s, X_1, \dots, X_b] & \textrm{ on } \bP^{b + s}_k, \\
            [Y_0, X_1, \dots, X_b] &\textrm{ on } \bP^{b}_k,
        \end{align*}
        so that the hyperplane $H$ is given by $\{Y_0 = 0\} \subset \bP^{s}_K$, and the projection is given by 
        \begin{equation*}
        \bP^{b + s}_k \dra \bP^b_k \colon [Y_0,Y_1, \dots, Y_s, X_1, \dots, X_b] \mapsto [Y_0, X_1, \dots, X_b].
        \end{equation*}
        Moreover, we define the hyperplanes $H' = \{Y_0 = 0\} \subset \bP^{b + s}_k$ and $H'' = \{Y_0 = 0\} \subset \bP^b_k$.

        Let $Z \subset \bP_k^{b+s}$ be the closure of $Q\cap U_H$, and let $Y$ be its strict transform in $\mathrm{Bl}_{\Lambda}\bP_k^{b+s}$. Write $f:Y\to\bP_k^b$ for the induced morphism.
        Choose a smooth alteration $\tau:\tilde Y\to Y$ of odd degree, which exists by Gabber's version of de Jong's alteration theorem \cite[Theorem~2.1]{IllusieTemkin2014}, and let $\rho:\tilde Y\to Z$ be the composition.
        We make two observations:
        \begin{enumerate}
            \item The generic fiber of $f$ is $Q$.
            \item \label{obs:int-with-exc}
            If $E\subset Y$ is the intersection of $Y$ with the exceptional divisor of $\mathrm{Bl}_{\Lambda} \bP_k^{b + s}$, then its generic fiber over $\bP_k^b$ is $Q\cap H$. This follows directly from local computation on the blowup, as in \cite[\S 7.1]{schreieder}.
        \end{enumerate}

        Let $\beta$ be the pullback of $\alpha$ to $K(Q)$, which is nonzero and unramified from the definition of a Schreieder triple, 
        and let $X$ be a very general hypersurface of degree $d' \geq d$. 
        We degenerate $X$ to the fiber $Z \cup (d' - d)H'$, where $H'$ is the hyperplane $\{Y_0 = 0\} \subset \bP^{b + s}_k$ defined above, and the notation $(d' - d)H'$ means that $H'$ appears in the special fiber with multiplicity $d' - d$.
        (If $d' = d$, then the special fiber is just $Z$.)
        By the version of the specialization method from \cite[Proposition 6.1]{schreieder-torsion}, it suffices to show that for any integral subvariety $T\subset\rho^{-1}(Z^{\mathrm{sing}}\cup(Z\cap H^{\prime}))$, we have the vanishing result $(\tau^* \beta)|_{T} = 0$.
        If $T$ does not dominate $\bP^b_k$, then we are done by \cite[Theorem 9.2]{schreieder}, so we assume that $T$ dominates $\bP^b_k$.
        
        Let $\nu \in T$ be the generic point. 
        Since $T$ dominates $\bP^b_k$, the image $\tau(\nu)$ lies in the generic fiber of $Y \to \bP^b_k$, which is simply the quadric $Q$.
        We claim that, in fact, $\tau(\nu)$ lies in $Q \cap H$, which coincides with $Q \cap E$ by Observation~\ref{obs:int-with-exc} above.
        
        Assume for the sake of contradiction that $\tau(\nu) \notin Q \cap E$.
        Outside of $E$, $\pi \colon Y \to Z$ is an isomorphism. 
        Since $Y$ is smooth at each point of its generic fiber $Q$, $\rho(\nu) = \pi(\tau(\nu))$ cannot lie in $Z^{\mathrm{sing}}$.
        Therefore, $\rho(\nu)$ lies in $H'$.
        But outside of the center of the blowup, the projection from $\bP^{b + s}_k$ to $\bP^b_k$ sends $H'$ onto the hyperplane $H'' \subset \bP^b_k$, contradicting the fact that $\nu$ dominates $\bP^b_k$.

        Therefore, $\tau(\nu)$ is contained in $Q \cap H$.
        From the definition of a Schreieder triple, the pullback of $\alpha$ to $K(Q \cap H)$ vanishes. 
        Since the pullback of $\alpha$ to $T$ factors through its pullback to $Q \cap H$ (and also the pullback to $Q$, where the class $\alpha$ becomes unramified), we conclude that $(\tau^* \beta)|_{T} = 0$ using the functoriality of unramified cohomology.
    \end{proof}

    \subsection{A criterion for unramifiedness}

    In this section, we recall an important criterion for unramifiedness (Proposition~\ref{prop:criterion-for-unramified} below), which has also been extracted from \cite{schreieder}.
    First, let us recall that if $\alpha = (\alpha_1, \dots, \alpha_b)$ is a symbol over a field $K$, then the Pfister form $\Pi$ associated to $\alpha$ is the quadratic form
    \begin{equation*}
        \Pi = \pfister{\alpha_1, \dots, \alpha_b} \coloneqq \langle 1, \alpha_1 \rangle \otimes \cdots \otimes \langle 1, \alpha_b \rangle.
    \end{equation*}
    We will frequently use the following general facts about the relationship between the symbol $\alpha$ and its Pfister form $\Pi$:

    \begin{theorem}
    \label{thm:pfister-symbol}
        Let $K$ be a field containing $k$, let $\alpha=(a_1,\dots,a_n)\in\rH^n(K,\bZ/2)$, and let $\Pi$ be its associated Pfister form.
        \begin{enumerate}
            \item $\Pi$ is isotropic if and only if $\alpha=0$.
            \item Suppose that $\Pi$ is anisotropic, and let $Q=\{q=0\}\subset\bP_K^s$ be a smooth irreducible projective quadric. Then the following are equivalent:
            \begin{enumerate}[label=(\roman*)]
                \item $\alpha_{K(Q)}=0$; 
                \item $\Pi_{K(Q)}$ is isotropic;
                \item $q$ is similar to a subform of $\Pi$.
            \end{enumerate}
        \end{enumerate}
    \end{theorem}

    \begin{proof}
        See \cite[Theorems~2.2 and 2.3]{schreieder}.
        As explained in loc. cit., (1) is a consequence of the Milnor conjecture, proved by Voevodsky in \cite{Voevodsky2003}.
    \end{proof}

    \begin{proposition}
    \label{prop:criterion-for-unramified}
        Let $K/k$ be a finitely generated field extension, let $\alpha \in \rH^n(K, \bZ/2)$ be a symbol with associated Pfister form $\Pi$, and let $Q = \{q = 0\}$ be a smooth integral projective quadric.  
        Suppose that for each discrete valuation $v$ on $K$ at which $\alpha$ ramifies, the form $q$ becomes similar to a subform of $\Pi$ over $K_v$.
        Then the pullback of $\alpha$ to $K(Q)$ is unramified over $k$.
    \end{proposition}

     \begin{proof}     
        Let $w$ be a discrete valuation on $K(Q)$.
        If $w|_K$ is trivial, every entry of the symbol $\alpha$ lies in $\cO^\times_w \subset K(Q)$, so $\alpha$ is unramified at $w$.

        Otherwise, if $w|_K$ is nontrivial, let $v$ be the restriction of $w$ to $K$.
        The inclusion $K \subset K(Q)$ extends to an inclusion between completions, $K_v\subset K(Q)_w$.
        If $\alpha$ is unramified at $v$, then its pullback is unramified at $w$ by the pullback formula for the residue, cf. \cite[\S 3.1, (3.3)]{schreieder-survey}.
        
        Therefore, we may suppose that $\alpha$ is ramified at $v$. 
        We make two observations:
        First, by assumption, $q$ is similar to a subform of $\Pi$ over $K_v$, and hence also over $K(Q)_w$.
        Second, $q$ is tautologically isotropic over $K(Q)$, so $q$ is isotropic over $K(Q)_w$.
        Putting these two observations together, we conclude that $\Pi$ is isotropic over $K(Q)_w$.
        Thus, Theorem~\ref{thm:pfister-symbol}(1) gives that $\alpha$ vanishes over $K(Q)_w$.

        The compatibility of residues with completion
        \cite[\S 3.1, Equation (3.3)]{schreieder-survey}
        therefore gives $\partial_w(\alpha_{K(Q)})=0$.
        Since $w$ was arbitrary, $\alpha_{K(Q)}$ is unramified over $k$.
\end{proof}

\subsection{Augmented subforms}
\label{ssec:augmented-subform}

We now describe an elementary construction which will be our main tool.
Let $K/k$ be a finitely generated field extension, and let $\alpha \in \rH^n(K, \bZ/2)$ be a symbol with associated Pfister form $\Pi$.
Suppose that $\sigma \subset \Pi$ is a subform. 

We consider the augmented subform $\langle 1 \rangle \perp \sigma$.
While $\sigma$ embeds into $\Pi$, the same is not necessarily true for $\langle 1 \rangle \perp \sigma$.

However, one has an embedding in the following simple situation. Let $\sigma^\perp$ denote the restriction of $\Pi$ to the orthogonal complement of the subspace supporting $\sigma$. 
If $\sigma^\perp$ represents $1$, choose a vector $u$ in this complement with $\Pi(u)=1$.
It follows that there is an embedding 
\begin{equation*}
    \langle 1 \rangle \perp \sigma \to \Pi
\end{equation*}
given by the identity on $\sigma$, and by sending the standard basis vector of the first summand to $u$.
Combining this observation with Proposition~\ref{prop:criterion-for-unramified}, we immediately obtain the following result:

\begin{corollary}
\label{cor:rep-one}
    Let $K/k$ be a finitely generated field extension, and let $\alpha \in \rH^n(K, \bZ/2)$ be a symbol with associated Pfister form $\Pi$.
    Let $\sigma \subset \Pi$ be a subform of dimension $\geq 2$, and let $Q$ be the smooth projective quadric associated to $\langle 1 \rangle \perp \sigma$.
    Suppose that for each discrete valuation $v$ on $K$ where $\alpha$ ramifies, $\sigma^\perp$ represents $1$ over the completion $K_v$.
    Then the pullback of $\alpha$ to $K(Q)$ is unramified.
\end{corollary}

    \section{The construction}
    \label{sec:construction}

    \subsection{Setup}
    
    Fix an integer $r \geq 3$. 
    We work over the field $K_r = k(z_1, \dots, z_r)$.
    (The assumption that $r \geq 3$ will not be used until \S\ref{ssec:define-pj} below, when we make a definition which is nonsensical for $r \leq 2$.) 
    Over this field, we consider the symbol
    \begin{equation*}
        \alpha_r = (z_1, \dots, z_r) \in \rH^r(K_r, \bZ/2), 
    \end{equation*}
    as well as its associated Pfister form $\Pi_r = \pfister{z_1, \dots, z_r}$.
    We write $\{e_I\},$ $I \subset \{1, \dots, r\}$, for the standard basis of the vector space underlying the Pfister form $\Pi_r$.
    We begin with a trivial remark, which will be used repeatedly:

    \begin{remark}
    \label{rem:ramified-val-of-zj}
        Let $v$ be a discrete valuation on $K_r$ such that $\alpha_r$ is ramified at $v$.
        Then there exists some $j \in \{1, \dots, r\}$ such that $v(z_j) \neq 0$.
        Indeed, if not, then $\alpha_r$ would be unramified, since each entry $z_j$ of $\alpha_r$ would be a unit in the valuation ring of $v$, cf. \S \ref{ssec:the-definition}.
    \end{remark}

    We start with the form $\sigma_0 = \langle 1, z_1, \dots, z_r \rangle$, which we embed into $\Pi_r$ using the vectors $e_{\emptyset}, e_1, \dots, e_r$, where we write $e_j = e_{\{j\}}$ for simplicity. 
    In other words, we view $\sigma_0$ as a subform of $\Pi_r$, supported on the vector space
    \begin{equation*}
        W_0 = \langle e_{\emptyset}, e_1, \dots, e_r \rangle.
    \end{equation*}
    Our goal is to introduce a perturbation $\sigma_t \subset \Pi_r$ of the subform $\sigma_0$ that satisfies the requirements of Corollary~\ref{cor:rep-one}.
    
    In particular, we will eventually need to write down some explicit elements of $\sigma_t^\perp$.
    To that end, it will turn out to be convenient to consider an alternative basis $v_0, \dots, v_r$ of $W_0$ satisfying the property that $\Pi_r(v_j^\vee)$ is a square in $K_r$ for each $j$.
    We define the vectors $v_0, \dots, v_r$ by
        \begin{align*}
        v_j &= \begin{cases}
            z_1 e_\emptyset  + \frac{z_1 - 1}{2} e_1 + \frac{z_1}{2} \sum_{i = 2}^r e_i & j = 0 \\
            \frac{z_{j + 1}}{2} e_j - \frac{z_j}{2} e_{j + 1} & 1 \leq j < r \\
            \frac{1}{2} e_r & j = r.
        \end{cases} 
    \end{align*}
    It is easy to see that $v_0, \dots, v_r$ form a basis for $W_0$, and we now explain the fact that $\Pi_r(v_j^\vee)$ is a square:

    \begin{lemma}
    \label{lem:square-norm}
    Let $v_0^\vee, \dots, v_r^\vee \in W_0$ be the dual basis of the basis $v_0, \dots, v_r$ introduced above.
    Then $\Pi_r(v_j^\vee) = c_j^2$ for each $0 \leq j \leq r$, where
    \begin{equation*}
            c_j = \begin{cases}
               z_1^{-1} & j = 0 \\
                z_j^{-1}z_{j + 1}^{-1}(1 + z_1+ \cdots + z_j) & 1 \leq j \leq r - 1 \\
                z_r^{-1}(1 + z_1+ \cdots + z_r) & j = r.
            \end{cases}
        \end{equation*}
    \end{lemma}

    \begin{proof}
       Writing $A_j = \sum_{i = 1}^j z_i$, a direct computation shows that
        \begin{align*}
            v_0^\vee &= \frac{1}{z_1} e_{\emptyset}, \\
            v_j^\vee &= \frac{1}{z_jz_{j + 1}}\left((1 - A_j)e_{\emptyset} + 2 \sum_{i = 1}^j e_i\right), \quad (1 \leq j < r) \\
            v_r^\vee &= \frac{1}{z_r} \left(  (1 - A_r) e_\emptyset + 2 \sum_{i = 1}^r e_i \right).
        \end{align*}
        The lemma follows easily.
    \end{proof}

    \subsection{Perturbing the form}
    \label{ssec:define-pj}

    In the vector space underlying $\Pi_r$, we define the vectors $p_j$ as follows:
        \begin{align*}
        p_j &= \begin{cases}
            e_{\{1,2,3\}} & j = 0 \\
            e_{\{j, j + 1\}} & 1 \leq j < r \\
            e_{\{1, r\}} & j = r.
        \end{cases}
    \end{align*}
    Note that each $p_j$ is orthogonal to all of the vectors $v_0, \dots, v_r$.
    In addition, we have used our assumption that $r \geq 3$, since otherwise it does not make sense to consider $e_{\{1,2,3\}}$ as a vector in the space underlying $\Pi_r$.

    \begin{definition}
        For a parameter $t \in k$, we define  $\sigma_t \subset \Pi_r$ as the subform supported on
        \begin{equation*}
            W_t \coloneqq \mathrm{Span}\{v_0 + t p_0, \dots, v_r + tp_r\},
        \end{equation*}
        where $v_j$, $p_j$ have been defined above.
        We observe that projection onto $W_0$ along $W_0^\perp$ sends $v_j+tp_j$ to $v_j$, so $\{v_j + tp_j\}$ indeed forms a basis of $W_t$, and that for $t = 0$, we recover the form $\sigma_0$ defined previously.
        Since the Pfister form $\Pi_r$ is anisotropic, the restriction of $\Pi_r$ to the subspace $W_t$ is an anisotropic quadratic form, and we define
        \begin{equation*}
            \sigma_t = \Pi_r|_{W_t}.
        \end{equation*}
        We define $Q_t \subset \bP^{r + 1}_{K_r}$ to be the smooth projective quadric associated to the form $\langle 1 \rangle \perp \sigma_t$,
        i.e., as the quadric arising from the augmented subform construction from \S \ref{ssec:augmented-subform} above.
    \end{definition}

    The main result of this section is the following:

    \begin{theorem}
    \label{thm:sigma-t-is-schr}
        Suppose that $t \in k$ is very general and $r \geq 3$, and let $H \subset \bP^{r + 1}_{K_r}$ be the hyperplane corresponding to the orthogonal complement of the first summand $\langle 1 \rangle$ of $\langle 1 \rangle \perp \sigma_t$.
        Then $(\alpha_r, Q_t, H)$ is a Schreieder triple. 
    \end{theorem}

    \begin{remark}[Even vs. odd dimension]
    \label{rem:even-vs-odd}
        The Schreieder triples in Theorem~\ref{thm:sigma-t-is-schr} result in hypersurfaces of dimension $2r$ over $k$, since in the end one gets a hypersurface in $\bP_k^{2r + 1}$.
        To cover the case of odd-dimensional hypersurfaces, we will eventually consider a slight modification of the construction, where we replace $\sigma_t$ by a larger form $\sigma_t^+$ in which $\sigma_t$ has codimension $1$.
        Since the proofs in the two cases $\sigma_t$, $\sigma_t^+$ are identical, we have opted to stick to the $\sigma_t$ case to avoid clutter, and leave the discussion of $\sigma_t^+$ for \S \ref{ssec:odd-dimension} below.
    \end{remark}

    In order to apply Theorem~\ref{thm:sigma-t-is-schr} to the stable irrationality of hypersurfaces of degree $\geq 5$, we also need some control on the degree of the hypersurface which results from $(\alpha_r, Q_t, H)$ in Theorem~\ref{thm:hypersurfaces-from-schr-trip}.
    It turns out that all we need to know is the following, which we record here for future reference:

    \begin{lemma}
    \label{lem:gram-matrix}
        In the basis $v_0 + t p_0, \dots, v_r + tp_r$ of $W_t$, the Gram matrix of $\sigma_t$ has entries which are polynomials of degree $\leq 3$ in the variables $z_1, \dots, z_r$.
    \end{lemma}

    \begin{proof}
        Let $M_t$ be the Gram matrix of $\sigma_t$.
        Since $p_0, \dots, p_r$ is an orthogonal set of vectors, and moreover is orthogonal to $W_0$, we see that 
        \begin{equation*}
            M_t = M_0 + t^2 M',
        \end{equation*}
        where $M'$ is the Gram matrix of the restriction of $\Pi_r$ to the subspace spanned by $p_0, \dots, p_r$.
    
        Recall the basis $e_I$ for the vector space underlying $\Pi_r$.
        Writing $\langle -,- \rangle$ for the bilinear form associated to $\Pi_r$, we have by definition
        \begin{equation}
        \label{eq:bil-pairing-pfister}
            \langle e_I, e_J \rangle = \begin{cases}
                \prod_{i \in I} z_i  & I = J \\
                0 & I \neq J.
            \end{cases}
        \end{equation}
        From this, it follows immediately that $M'$ is a matrix of polynomials of degree $\leq 3$ in $z_1, \dots, z_r$.
        Indeed, $p_j$ is equal (by definition) to $e_I$ for some $I$ with $|I| \leq 3$.
        Hence, to conclude the lemma, we just need to show that $M_0$ is a matrix of polynomials of degree $\leq 3$.

        For $M_0$, observe from the definition that each $v_j$ is a linear combination of $e_\emptyset$, $e_1, \dots, e_r$ with  coefficients which are polynomials of degree at most $1$ in $z_1, \dots, z_r$.
        Since the basis $e_{\emptyset}, e_1, \dots, e_r$ is orthogonal, we see from \eqref{eq:bil-pairing-pfister} that $\langle v_i, v_j \rangle$ is a polynomial of degree $\leq 3$ in $z_1, \dots, z_r$.
    \end{proof}

    \subsection{Some vectors in the complement}

    Throughout this subsection, assume that $t \neq 0$. 
    Under this assumption, we describe some vectors $s_0, \dots, s_r$ in the orthogonal complement $W_t^\perp$ with respect to $\Pi_r$.
    The vectors $s_0, \dots, s_r$ will be used to test whether $\sigma_t^\perp$ represents $1$ over a completion $K_{r, v}$ of $K_r$.

    \begin{definition}
        We define the vectors $s_0, \dots, s_r$ by 
        \begin{equation*}
            s_j = \frac{1}{c_j} \left( v_j^\vee - \frac{1}{t \cdot \Pi_r(p_j)} p_j  \right).
        \end{equation*}
    \end{definition}

    The only properties which we will need from $s_0, \dots, s_r$ are the following ones:

    \begin{lemma}
    \label{lem:norm-of-sj}
        The vectors $s_j$ belong to $W_t^\perp$, and moreover 
        \begin{equation*}
            \Pi_r(s_j) = 1 + \frac{1}{t^2 \Pi_r(v_j^\vee)\Pi_r(p_j)}
        \end{equation*}
        for each $0 \leq j \leq r$.
    \end{lemma}

    \begin{proof}
        For each $j$, one easily checks by pairing $s_j$ with every $v_\ell+tp_\ell$ that $s_j\in W_t^\perp$.
        The formula for $\Pi_r(s_j)$ is a direct computation using the description of the vectors $v_j^\vee$ from the proof of Lemma~\ref{lem:square-norm}.
    \end{proof}

    \begin{corollary}
    \label{cor:sigma-rep-one}
        Suppose that $t \in k^\times$. 
        Let $v$ be a discrete valuation on $K_r$, and suppose that 
        \begin{equation*}
            v(\Pi_r(v_j^\vee)\Pi_r(p_j)) < 0
        \end{equation*}
        for some $0 \leq j \leq r$.
        Then $\sigma_t^\perp$ represents $1$ over $K_{r,v}$.
    \end{corollary}

    \begin{proof}
        Choose $j$ so that $v(\Pi_r(v_j^\vee)\Pi_r(p_j)) < 0$. 
        It follows from Lemma~\ref{lem:norm-of-sj} that $\Pi_r(s_j) = 1 + x$, where $v(x) > 0$.
        In particular, by Hensel's lemma, $\Pi_r(s_j)$ admits a square root, say $c$, over the completion $K_{r, v}$.
        Then $\sigma^\perp_t(s_j/c) = \Pi_r(s_j/c) = 1$, so $\sigma_t^\perp$ represents $1$ over $K_{r, v}$.
    \end{proof}

    \subsection{Unramifiedness}

    \begin{theorem}
    \label{thm:unramified}
        For every $t \in k^\times$, the pullback of $\alpha_r$ to $Q_t$ is unramified. 
    \end{theorem}

    Using our previous results, we quickly reduce the theorem to the following proposition:

    \begin{proposition}
    \label{prop:key-prop}
        Let $v$ be a discrete valuation on $K_r$ at which $\alpha_r$ is ramified. 
        Then there exists $j$ with $0 \leq j \leq r$ such that $v(\Pi_r(v_j^\vee)\Pi_r(p_j)) < 0$.
    \end{proposition}

    \begin{proof}[Proof that Proposition~\ref{prop:key-prop} implies Theorem~\ref{thm:unramified}]
        By Corollary~\ref{cor:rep-one}, it suffices to show that for every discrete valuation $v$ on $K_r$ at which $\alpha_r$ is ramified, the form $\sigma^\perp_t$ represents $1$ over $K_{r,v}$.
        By Corollary~\ref{cor:sigma-rep-one}, $\sigma_t^\perp$ represents $1$ over $K_{r,v}$ if $v(\Pi_r(v_j^\vee)\Pi_r(p_j)) < 0$.
    \end{proof}

    \begin{proof}[Proof of Proposition~\ref{prop:key-prop}]
        Assume for the sake of contradiction that $v(\Pi_r(p_j)\Pi_r(v_j^\vee)) \geq 0$ for all $j$.
        A direct computation shows that
        \begin{equation}
        \label{eq:product-of-norms}
            \Pi_r(p_j)\Pi_r(v_j^\vee) = \begin{cases}
                z_1^{-1} z_2z_3 & j = 0 \\
                z_j^{-1}z_{j + 1}^{-1}(1 + z_1 + \cdots + z_j)^2 & 1 \leq j \leq r- 1 \\
                z_1z_r^{-1}(1 + z_1 + \cdots + z_j)^2 & j = r.
            \end{cases} 
        \end{equation}
        By assumption, the valuation of any expression on the right-hand side is nonnegative; using this, we will form a system of inequalities involving the $v(z_j)$ whose only solution is the trivial one, $0 = v(z_1) = \cdots = v(z_r)$, contradicting the assumption that $\alpha_r$ is ramified by Remark~\ref{rem:ramified-val-of-zj}.

        To carry out this plan, we must first analyze the valuation of $1 + z_1 + \cdots + z_j$.
        In fact, we claim that
        \begin{equation}
        \label{eq:sum-vs-min-val}
            v(1 + z_1 + \cdots + z_j) = \min\{0, v(z_1), \dots, v(z_j)\}.
        \end{equation}
        To prove this, by a simple argument (included here as Lemma~\ref{lem:min-valuation} below), it suffices to show that $\langle 1, z_1, \dots, z_j \rangle$ is anisotropic over $K_{r, v}$, for which it suffices to prove anisotropy of $\langle 1,z_1,\dots,z_r\rangle$.
        But $\langle 1, z_1, \dots, z_r \rangle$ embeds into $\Pi_r$ as the form $\sigma_0$ discussed above.
        Since $\alpha_r$ is ramified at $v$ by assumption, its Pfister form $\Pi_r$ is anisotropic over $K_{r, v}$ by Theorem~\ref{thm:pfister-symbol}, so $\sigma_0$ is anisotropic. 

        Using \eqref{eq:sum-vs-min-val}, we now form the system of inequalities from the right-hand side of \eqref{eq:product-of-norms}.
        We first examine the case $1 \leq j \leq r - 1$ of \eqref{eq:product-of-norms}.
        Applying the valuation and using our nonnegativity assumption $v(\Pi_r(p_j)\Pi_r(v^\vee_j)) \geq 0$, we obtain 
        \begin{equation*}
            v(z_j) + v(z_{j + 1}) \leq 2 \min\{0, v(z_1), \dots, v(z_j)\} \leq 2 v(z_j).
        \end{equation*}
        From this, one easily sees (by considering both possibilities for the sign of $v(z_1)$) that
        \begin{equation}
        \label{eq:chain-of-norms}
            v(z_r) \leq v(z_{r - 1}) \leq \cdots \leq v(z_2) \leq -|v(z_1)|,
        \end{equation}
        and moreover  that
        \begin{equation*}
            \min\{0, v(z_1), \dots, v(z_r)\} = v(z_r).
        \end{equation*}

        Next, we apply the valuation  to the $j = r$ case of \eqref{eq:product-of-norms}:
        \begin{equation*}
            v(z_1)  - v(z_r) +  2 \min\{0, v(z_1), \dots, v(z_r)\} = v(z_1) + v(z_r) \geq 0,
        \end{equation*}
        so that $v(z_r) \geq -v(z_1)$.
        Returning to \eqref{eq:chain-of-norms} (and again taking into account both possibilities for the sign of $v(z_1)$), we conclude that in fact $v(z_r) = -v(z_1)$, so that 
        \begin{equation}
        \label{eq:equality-of-norms}
            v(z_r) = v(z_{r - 1}) = \cdots = v(z_2) = -v(z_1) , \quad v(z_1) \geq 0.
        \end{equation}

        Finally, using the $j = 0$ case of \eqref{eq:product-of-norms}, we find that
        \begin{equation*}
                v(z_2) + v(z_3) - v(z_1) \geq 0.
        \end{equation*}
        By \eqref{eq:equality-of-norms}, we get that $-3v(z_1) \geq 0$.
        From the previous paragraph, we know that $v(z_1) \geq 0$, so it follows that $v(z_1) = 0$.
        By \eqref{eq:equality-of-norms}, we obtain that $v(z_j) = 0$ for all $j$.
        By Remark~\ref{rem:ramified-val-of-zj}, this contradicts the assumption that $\alpha_r$ is ramified at $v$.
    \end{proof}

    \begin{lemma}
    \label{lem:min-valuation}
        Let $L$ be a field which is complete with respect to a discrete valuation $v$ of residue characteristic different from $2$, and suppose that $q = \langle 1, a_1,\dots, a_n \rangle$ is an anisotropic quadratic form over $L$.
        Then 
        \begin{equation*}
            v(1 + a_1 + \cdots + a_n) = \min\{0, v(a_1), \dots, v(a_n)\}.
        \end{equation*}
    \end{lemma}

    \begin{proof}
        Let $S = 1 + a_1 + \cdots + a_n$, and choose $j$ such that $v(a_j) = \min \{v(a_0), v(a_1), \dots, v(a_n)\}$, where by convention $a_0 = 1$.
        Assuming $v(S) > v(a_j)$, we will show that $q$ is isotropic.
        Under our assumption,
        \begin{equation*}
            1 - \frac{S}{a_j} = 1 \mod \fm,
        \end{equation*}
        where $\fm$ is the maximal ideal of the valuation ring inside of $L$.
        By Hensel's lemma, we see that $1 - S/a_j$ admits a square root, say $c$, in $L$.
        But then if we define the vector $x = (1, \dots, 1, c, 1, \dots, 1)$, where $c$ occurs at the $j$th slot, we get
        \begin{equation*}
            q(x) = \sum_{i \neq j} a_i + c^2 a_j = S + (c^2 - 1)a_j = 0.
        \end{equation*}
        In particular, $x$ is an isotropic vector, so $q$ is isotropic.
    \end{proof}

    \subsection{Proof of Theorem~\ref{thm:sigma-t-is-schr}}
    \label{ssec:proof-of-unram}

    By Theorem~\ref{thm:unramified}, the pullback of $\alpha_r$ to $K_r(Q_t)$ is unramified when $t$ is nonzero. 
    On the other hand, the pullback of $\alpha_r$ to $K_r(Q_t \cap H)$ vanishes by Theorem~\ref{thm:pfister-symbol}, since $Q_t \cap H$ is the quadric associated to the quadratic form $\sigma_t$, and $\sigma_t$ is a subform of $\Pi_r$.
    (Note that the definition of a Schreieder triple requires that $Q_t \cap H$ is smooth and integral; smoothness is automatic since $\sigma_t$ is nondegenerate, while integrality holds because $Q_t \cap H$ is smooth and has positive dimension.)
    Therefore, to prove that $(\alpha_r, Q_t, H)$ is a Schreieder triple, it remains to show that the pullback of $\alpha_r$ to $K_r(Q_t)$ is nonzero when $t$ is very general. 

    After a standard spreading out argument, a result of Gabber \cite[Appendix, Proposition A7]{CT-examples} reduces nonvanishing for very general $t$ to the case $t=0$.
    Since $\sigma_0 = \langle 1, z_1, \dots, z_r \rangle$, we see that $\langle 1 \rangle \perp \sigma_0$ is isotropic, so $K_r(Q_0)$ is purely transcendental over $K_r$.
    In particular, the pullback map on \'etale cohomology is injective, hence the pullback of $\alpha_r$ to $K_r(Q_0)$ remains nonzero.
    \qed

    \subsection{The odd-dimensional case}
    \label{ssec:odd-dimension}

    As indicated above in Remark~\ref{rem:even-vs-odd}, Theorem~\ref{thm:sigma-t-is-schr} eventually gives results only for even-dimensional hypersurfaces. 
    In this subsection, we explain a slight modification of the construction which covers the case of odd-dimensional hypersurfaces.

    \begin{definition}
        Assuming that $r \geq 4$, it is easy to see that the form $\sigma_t$ is orthogonal to the vector $e_{\{1,3\}}$ in the space underlying $\Pi_r$, and so we may define the subform
        \begin{equation*}
            \sigma_t^+ = \sigma_t \perp \langle z_1z_3 \rangle \subset \Pi_r,
        \end{equation*}
        where $\langle z_{1}z_{3} \rangle$ embeds into $\Pi_r$ via $e_{\{1,3\}}$.
        In addition, let $Q_t^+$ be the quadric associated to the augmented subform $\langle 1 \rangle \perp \sigma_t^+$.
        (If $r = 3$, then $p_r = e_{\{1,3\}}$, so one loses the orthogonality.)
    \end{definition}

    \begin{theorem}
    \label{thm:odd-dim-schr-triple}
        Let $r \geq 4$. Let $H \subset \bP_{K_r}^{r + 2}$ be the hyperplane corresponding to the orthogonal complement of $\langle 1 \rangle$. 
        For $t$ very general, the triple $(\alpha_r, Q_t^+, H)$ is a Schreieder triple.
    \end{theorem}

    \begin{proof}
        The key point is that the vectors $s_0, \dots, s_r$ are orthogonal to $\sigma_t^+$.
        Proving this just requires one to check that $s_0, \dots, s_r$ are orthogonal to $e_{\{1,3\}}$, which one can check directly.
        Once this is known, then the fact that the pullback of $\alpha_r$ to $K_r(Q_t^+)$ is unramified follows by an identical argument as in the case of $K_r(Q_t)$.
        
        The hyperplane section $Q_t^+ \cap H$ is again a smooth integral quadric, and one can prove that the pullback of $\alpha_r$ to $K_r(Q_t^+ \cap H)$ vanishes by an argument identical to the one given in \S \ref{ssec:proof-of-unram}. 
        For nonvanishing of the pullback of $\alpha_r$ to $K_r(Q_t^+)$, we reduce (by arguing as in \S \ref{ssec:proof-of-unram}) to the case $t = 0$.
        In this case, we again use the fact that the quadric $Q_0^+$ is rational over $K_r$ (since the quadratic form $\langle 1 \rangle \perp \sigma_0^+$ contains $\langle 1,1 \rangle$ as a summand), and the pullback of a nonzero cohomology class (e.g., $\alpha_r$) along a purely transcendental extension remains nonzero.
    \end{proof}

    \section{Conclusion}
    \label{sec:conclusion}

    In this section, we finish the proof of Theorem~\ref{thm:main-intro} from the introduction.

    \begin{theorem}
        Let $k$ be an uncountable algebraically closed field of characteristic not $2$. Let $X$ be a very general hypersurface of degree $d \geq 5$ in $\bP_k^{n + 1}$ for $n \geq 1$.
        Then $X$ does not admit a decomposition of the diagonal.
        In particular, $X$ is stably irrational.
    \end{theorem}

    \begin{proof}
        First, the cases $n = 1$ and $2$ are elementary, and follow (for example) from the fact that a smooth projective variety $X$ with a decomposition of the diagonal satisfies $\rH^0(X, \Omega^i_X) = 0$ for $i > 0$, cf. \cite[Lemma 2.2]{Totaro2016}.
        The cases $n \in \{3, \dots, 9\}$ are covered by Schreieder's bound from \cite[Theorem 8.1]{schreieder}, so it suffices to consider the case when $n \geq 10$.
        In fact, our construction works for $n = 6$ and $n \geq 8$, because these are the integers which may be written as $2r$ for $r \geq 3$, or $2r + 1$ for $r \geq 4$.
        
        Fix $t \in k$ very general. We just need to show that the Schreieder triples $(\alpha_r, Q_t, H)$ and $(\alpha_r, Q_t^+, H)$ from Theorems~\ref{thm:sigma-t-is-schr} and \ref{thm:odd-dim-schr-triple} lead to bounds for hypersurfaces of degree $\geq 5$ via Theorem~\ref{thm:hypersurfaces-from-schr-trip}.
        Let $M_t = (m_{ij})$ be the Gram matrix of $\sigma_t$ in the basis $v_j + tp_j$.
        The quadrics $Q_t$ and $Q_t^+$ are given by
        \begin{equation}
        \label{eq:equations-for-qt-etc}
            X^2 + \sum_{i,j} m_{ij} Y_i Y_j = 0, \quad X^2 + \sum_{i,j} m_{ij} Y_iY_j + z_1z_3 Z^2 = 0
        \end{equation}
        in $\bP^{r + 1}_{K_r}$ and $\bP^{r + 2}_{K_r}$, respectively, where the hyperplane $H$ is given by $X = 0$ in both cases.
        
        By Lemma~\ref{lem:gram-matrix}, the coefficients $m_{ij}$ are polynomials of degree $\leq 3$ in the variables $z_1, \dots, z_r$,
        so on the chart $X = 1$, the equations \eqref{eq:equations-for-qt-etc} have total degree $\leq 5$ in the base variables $z_1, \dots, z_r$ and the fiber variables $Y_0, \dots, Y_r, Z$.
        Thus the closures of $Q_t \setminus H$ and $Q_t^+ \setminus H$ in $\bP^{2r + 1}_k$ and $\bP^{2r + 2}_k$, respectively, have degrees $\leq 5$.
        We conclude using Theorem~\ref{thm:hypersurfaces-from-schr-trip}.
    \end{proof}

    \bibliography{references}

@article{CT-examples,
  author    = {Colliot-Th{\'e}l{\`e}ne, Jean-Louis},
  title     = {Exposant et indice d'alg{\`e}bres simples centrales
               non ramifi{\'e}es},
  journal   = {L'Enseignement Math{\'e}matique (2)},
  volume    = {48},
  number    = {1--2},
  pages     = {127--146},
  year      = {2002},
  note      = {With an appendix by Ofer Gabber},
  doi       = {10.5169/seals-66070}
}

@misc{EngelDeGaayFortmanSchreieder2025,
  author        = {Engel, Philip and de Gaay Fortman, Olivier
                   and Schreieder, Stefan},
  title         = {Matroids and the integral {Hodge} conjecture
                   for abelian varieties},
  year          = {2025},
  howpublished  = {arXiv:2507.15704},
  eprint        = {2507.15704},
  archivePrefix = {arXiv},
  primaryClass  = {math.AG},
  url           = {https://arxiv.org/abs/2507.15704}
}

@article{kollar1995,
  author    = {Koll{\'a}r, J{\'a}nos},
  title     = {Nonrational hypersurfaces},
  journal   = {Journal of the American Mathematical Society},
  volume    = {8},
  number    = {1},
  pages     = {241--249},
  year      = {1995},
  doi       = {10.1090/S0894-0347-1995-1273416-8}
}

@article{lange-schreieder,
  author    = {Lange, Jan and Schreieder, Stefan},
  title     = {On the rationality problem for low degree hypersurfaces},
  journal   = {Forum of Mathematics, Pi},
  volume    = {13},
  pages     = {e25},
  year      = {2025},
  doi       = {10.1017/fmp.2025.10016}
}

@misc{moe,
  author        = {Moe, Simen Westbye},
  title         = {On stable rationality of polytopes},
  year          = {2023},
  howpublished  = {arXiv:2311.01144},
  eprint        = {2311.01144},
  archivePrefix = {arXiv},
  primaryClass  = {math.AG},
  url           = {https://arxiv.org/abs/2311.01144}
}

@article{nicaise-ottem,
  author    = {Nicaise, Johannes and Ottem, John Christian},
  title     = {Tropical degenerations and stable rationality},
  journal   = {Duke Mathematical Journal},
  volume    = {171},
  number    = {15},
  pages     = {3023--3075},
  year      = {2022},
  doi       = {10.1215/00127094-2022-0065}
}

@article{schreieder,
  author    = {Schreieder, Stefan},
  title     = {Stably irrational hypersurfaces of small slopes},
  journal   = {Journal of the American Mathematical Society},
  volume    = {32},
  number    = {4},
  pages     = {1171--1199},
  year      = {2019},
  doi       = {10.1090/jams/928}
}

@article{schreieder-torsion,
  author    = {Schreieder, Stefan},
  title     = {Torsion orders of {Fano} hypersurfaces},
  journal   = {Algebra \& Number Theory},
  volume    = {15},
  number    = {1},
  pages     = {241--270},
  year      = {2021},
  doi       = {10.2140/ant.2021.15.241}
}

@incollection{schreieder-survey,
  author    = {Schreieder, Stefan},
  title     = {Unramified cohomology, algebraic cycles and rationality},
  editor    = {Farkas, Gavril and van der Geer, Gerard
               and Shen, Mingmin and Taelman, Lenny},
  booktitle = {Rationality of varieties},
  series    = {Progress in Mathematics},
  volume    = {342},
  pages     = {345--388},
  publisher = {Birkh{\"a}user},
  address   = {Cham},
  year      = {2021},
  doi       = {10.1007/978-3-030-75421-1_13}
}

@misc{schreieder-icm,
  author        = {Schreieder, Stefan},
  title         = {Rationality of hypersurfaces},
  year          = {2025},
  howpublished  = {arXiv:2510.13679},
  eprint        = {2510.13679},
  archivePrefix = {arXiv},
  primaryClass  = {math.AG},
  note          = {To appear in the Proceedings of the International
                   Congress of Mathematicians 2026},
  url           = {https://arxiv.org/abs/2510.13679}
}

@article{Totaro2016,
  author    = {Totaro, Burt},
  title     = {Hypersurfaces that are not stably rational},
  journal   = {Journal of the American Mathematical Society},
  volume    = {29},
  number    = {3},
  pages     = {883--891},
  year      = {2016},
  doi       = {10.1090/jams/840}
}

@article{Voisin2015,
  author  = {Voisin, Claire},
  title   = {Unirational threefolds with no universal codimension 2 cycle},
  journal = {Invent. Math.},
  volume  = {201},
  number  = {1},
  year    = {2015},
  pages   = {207--237},
  doi     = {10.1007/s00222-014-0551-y}
}

@article{ColliotThelenePirutka2016,
  author  = {Colliot-Th{\'e}l{\`e}ne, Jean-Louis and Pirutka, Alena},
  title   = {Hypersurfaces quartiques de dimension 3 :
             non-rationalit{\'e} stable},
  journal = {Ann. Sci. {\'E}c. Norm. Sup{\'e}r. (4)},
  volume  = {49},
  number  = {2},
  year    = {2016},
  pages   = {371--397},
  doi     = {10.24033/asens.2285}
}

@article{Schreieder2019QuadricBundles,
  author  = {Schreieder, Stefan},
  title   = {On the rationality problem for quadric bundles},
  journal = {Duke Math. J.},
  volume  = {168},
  number  = {2},
  year    = {2019},
  pages   = {187--223},
  doi     = {10.1215/00127094-2018-0041}
}

@misc{Ottem2026,
  author        = {Ottem, John Christian},
  title         = {Very general quartic sixfolds are not stably rational},
  year          = {2026},
  howpublished  = {arXiv:2609.10231},
  eprint        = {2609.10231},
  archivePrefix = {arXiv},
  primaryClass  = {math.AG},
  url           = {https://arxiv.org/abs/2609.10231}
}

@incollection{IllusieTemkin2014,
  author    = {Illusie, Luc and Temkin, Michael},
  title     = {{Gabber}'s modification theorem (log smooth case)},
  booktitle = {Travaux de {Gabber} sur l'uniformisation locale et la
               cohomologie {\'e}tale des sch{\'e}mas quasi-excellents.
               S{\'e}minaire {\`a} l'{\'E}cole polytechnique 2006--2008},
  editor    = {Illusie, Luc and Laszlo, Yves and Orgogozo, Fabrice},
  series    = {Ast{\'e}risque},
  number    = {363--364},
  pages     = {167--212},
  year      = {2014},
  publisher = {Soci{\'e}t{\'e} Math{\'e}matique de France},
  note      = {Expos{\'e} X},
  url       = {https://www.numdam.org/item/AST_2014__363-364__167_0/}
}

@article{CTO1989,
  author  = {Colliot-Th{\'e}l{\`e}ne, Jean-Louis and Ojanguren, Manuel},
  title   = {Vari{\'e}t{\'e}s unirationnelles non rationnelles:
             au-del{\`a} de l'exemple d'{Artin} et {Mumford}},
  journal = {Invent. Math.},
  volume  = {97},
  number  = {1},
  year    = {1989},
  pages   = {141--158},
  doi     = {10.1007/BF01850658}
}

@article{HPT2018,
  author  = {Hassett, Brendan and Pirutka, Alena and Tschinkel, Yuri},
  title   = {Stable rationality of quadric surface bundles over surfaces},
  journal = {Acta Math.},
  volume  = {220},
  number  = {2},
  year    = {2018},
  pages   = {341--365},
  doi     = {10.4310/ACTA.2018.v220.n2.a4}
}

@article{Voevodsky2003,
  author  = {Voevodsky, Vladimir},
  title   = {Motivic cohomology with {$\mathbf{Z}/2$}-coefficients},
  journal = {Publ. Math. Inst. Hautes {\'E}tudes Sci.},
  volume  = {98},
  year    = {2003},
  pages   = {59--104},
  doi     = {10.1007/s10240-003-0010-6}
}
    \bibliographystyle{alpha}

\end{document}